\documentclass[11pt]{article}
\usepackage{amsmath}
\usepackage{amssymb}
\usepackage{amsthm}
\usepackage{mathtools}

\usepackage[usenames,dvipsnames]{color}
\usepackage[%
   breaklinks,
   colorlinks=true,
   citecolor=blue,
   linkcolor=red,
   pagebackref=false,
   pdfauthor={V. Strauss, M. Winklmeier},
   pdftitle={On the one-dimensional harmonic oscillator with a singular perturbation},]{hyperref}

\newcommand\mylabelenumi[1]{{\upshape (#1)}}
\newcommand\enumiref[1]{\mylabelenumi{\ref{#1}}}

\def\s-lim{\mathop{\rm s-lim\,}}

\def\dim{\mathop{\rm dim\,}}

\def\supp{\mathop{\rm supp\,}}

\DeclareMathOperator{\linspan}{span}

\DeclareMathOperator{\im}{Im}

\newcommand{\e}{\mathrm{e}}
\newcommand{\I}{\mathrm{i}}
\newcommand{\rd}{\mathrm{d}}
\newcommand{\C}{\mathbb C}
\newcommand{\N}{\mathbb N}
\newcommand{\R}{\mathbb R}
\newcommand{\FA}{\mathfrak A} 	
\newcommand{\dop}{A} 		
\newcommand{\cop}{C} 		
\newcommand{\Bop}{B} 		
\newcommand{\mD}{\mathcal D}

\theoremstyle{theorem}
\newtheorem{theorem}{Theorem}
\newtheorem*{theorem*}{Theorem}
\newtheorem{corollary}[theorem]{Corollary}
\newtheorem*{corollary*}{Corollary}
\newtheorem{proposition}[theorem]{Proposition}
\newtheorem{lemma}[theorem]{Lemma}
\newtheorem*{lemma*}{Lemma}

\theoremstyle{definition}
\newtheorem{remark}[theorem]{Remark}
\newtheorem*{remark*}{Remark}

\begin{document}

\title{On the one-dimensional harmonic oscillator with a singular perturbation}
\author{}\date{\today}\maketitle
\abstract{
In this paper we investigate the one-dimensional harmonic oscillator with a singular perturbation concentrated in one point.
We describe all possible selfadjoint realizations and we show that for certain conditions on the perturbation exactly one negative eigenvalues can arise. This eigenvalue tends to $-\infty$ as the perturbation becomes stronger.}

\section{Introduction} 
The one-dimensional harmonic oscillator is given by the formal differential expression
\begin{align*}
   \FA f(t) := \left( -\frac{1}{2} \frac{\rd^2}{\rd t^2} + \frac{1}{2} t^2 \right)f(t).
\end{align*}

The aim of this paper is to investigate several possible realisations of $\FA$ as a symmetric linear operator and determine all possible selfadjoint extensions.

By the Liouville-Green asymptotic formula \cite[Theorem 2.2.1]{eastham}, it is known that there are solutions $y_\pm$ of $\FA y=\lambda y$ with the following asymptotic behaviour for $|t|\to\infty$:
\begin{equation}
   \label{eq:assymptotic}
   y_\pm(t) \sim \frac{1}{ (t-2\lambda)^{\frac{1}{2}} }
   \exp\Big( \pm\int_a^t \sqrt{ s^2 - 2\lambda + \frac{1}{4} \Big( \frac{ s }{s^2-2\lambda}\Big)^{2} }\ \rd s\Big)
\end{equation}
for $|a|$ large enough.
This shows immediately that $\FA$ is in the limit point case both at $+\infty$ and $-\infty$.

In order to assign an operator to the differential expression $\FA$, we need to specify a domain of admissible functions.
The \emph{minimal operator} associated with $\FA$ is
\begin{equation*}
   \dop^{\min} f = \FA f,
   \qquad
   \mD(\dop^{\min}) = C_c^\infty(\R).
\end{equation*}
Since $\FA$ is in the limit point case both at $+\infty$ and $-\infty$, the operator
$\dop^{\min}$ is essentially selfadjoint (see, e.g. \cite[7.1.3]{Triebel}).
Its closure is the so-called \emph{maximal operator} associated to $\FA$:
\begin{equation}
   \label{eq:A}
   \dop f = \FA f,
   \qquad
   \mD(\dop) = \{ f:\R\to\C : f, f'\ \text{abs. cont.},\ f,\,\FA f\in L_2(\R)\}.
\end{equation}
Note that
\begin{equation*}
   \dop^{\min} \subset \overline{\dop^{\min}} = \dop = \dop^*.
\end{equation*}
It is well known that $\dop$ has a compact resolvent,
and that its spectrum consists of simple eigenvalues:
\begin{equation}
   \label{eq:Aspectrum}
   \sigma(\dop)=\sigma_p(\dop)=\Big\{ n+\frac{1}{2} : n\in\N_0 \Big\}.
\end{equation}
The corresponding eigenfunctions are
\begin{equation*}
   \psi_n(t) = \e^{-t^2/2} c_n H_n(t)
\end{equation*}
where $H_n$ is the $n$th Hermite polynomial of order $n$,
\begin{equation*}
   H_n(t) = (-1)^n \e^{t^2} \frac{\rd^n}{\rd t^n} \e^{-t^2},
\end{equation*}
and the normalisation factor 
$c_n := (\pi^{\frac{1}{2}} 2^n n!)^{-\frac{1}{2}}$
is chosen such that $\langle\psi_n, \psi_m\rangle = \delta _{nm}$.
\smallskip

\begin{remark}
   A straightforward calculation shows that
   if $u$ is a solution of $(\FA+\lambda)u=0$, then
   $\left( t  - \frac{\rd}{\rd t} \right)^n u$ is a solution of
   $(\FA+\lambda-n)u=0$ and
   $\left( \frac{\rd}{\rd t} + t \right)^n u$ is a solution of
   $(\FA+\lambda+n)u=0$.

   In particular, all eigenfunctions of $\dop$ can be obtained by the recursion
   \begin{equation}
      \label{eq:recursion}
      \psi_0(t) = \pi^{-\frac{1}{2}} \e^{-t^2/2},
      \qquad
      \psi_n(t) = c_n \left( t - \frac{\rd}{\rd t} \right)^{n}\e^{-t^2/2}.
   \end{equation}
   Note that $\left( \frac{\rd}{\rd t} + t \right)^n \psi_0 =0$, in agreement with the fact that $\dop$ has no negative eigenvalues.
\end{remark}

From the recursion formula \eqref{eq:recursion} it is clear that
$\psi_n$ is an even function if $n$ is even, and that it is an odd function, if $n$ is odd.
In particular, we have for $n\in\N_0$
\begin{align*}
   \psi_{2n}(-x) &= \psi_{2n}(x), &\psi_{2n}(0)&\neq 0, &\psi_{2n}'(0)&=0,\\
   \psi_{2n+1}(-x) &= -\psi_{2n+1}(x), &\psi_{2n+1}(0) &= 0, &\psi_{2n+1}'(0) &\neq 0.
\end{align*}

In Section~\ref{sec:halfline} we consider the restriction of the harmonic oscillator to the open half lines $R_\pm$:
\begin{align*}
   \dop_{\pm}^{\min} f(t) &:= \FA f(t),
   & \mD(\dop_{\pm}^{\min}) &:= C_c^\infty(\R_\pm),
   \\[2ex]
   \dop_{\pm}^{\max} f(t) &:= \FA f(t),
   & \mD(\dop_{\pm}^{\max}) &:=
   \left\{ f:\R_\pm\to\C  \ :\
   \begin{aligned}
      &f, f'\ \text{abs. cont.},\\
      &\FA f|_{\R_\pm} \in L_2(\R_\pm)
   \end{aligned}
   \right\}.
\end{align*}
With them we define
$\dop_0 = \dop_{-}^{\min} \oplus \dop_{+}^{\min}$.
Clearly, $\mD(\dop_0^*) = \mD(\dop_{-}^{\max}) \oplus \mD(\dop_{+}^{\max})$.

In sections~\ref{sec:B} and \ref{sec:C} we will study several restrictions of the operator $\dop$ by imposing conditions at $t=0$.
We define the closed symmetric operators
\begin{align*}
   \Bop f &:= \FA f,
   \qquad
   & \mD(\Bop) &:= \{ f \in \mD(\dop) : f(0) = 0\},
   \\
   \cop f &:= \FA f,
   \qquad
   & \mD(\cop) &:= \{ f \in \mD(\dop) : f(0) = f'(0) = 0\}.
\end{align*}
So we have the following chain of operators
\begin{equation*}
   \dop_0 \subset \overline{\dop_0} = \cop \subset \Bop \subset \dop
   = \dop^* \subset \Bop^* \subset \cop^* = \dop_0^*.
\end{equation*}
With exception of the first one, all inclusions are one-dimensional.
We will classify all selfadjoint extensions of $\Bop$ and $\cop$ in terms of conditions on the behaviour at $0$ of the functions in the corresponding domains.
Slightly abusing language, we will call these conditions \emph{boundary conditions at $0$}.
We will not use the von Neumann extension theory for symmetric operators, but will identify selfadjoint extensions with maximal neutral subspaces of $\C^2$, $\C^3$ and $\C^4$, respectively, equipped with an inner product induced by the condition
$\langle \FA f,\ g\rangle = \langle f,\ \FA g\rangle$ for $f,g$ in appropriate spaces.
It turns out that the selfadjoint extensions of $\Bop$ can be parametrised by one real parameter. Every selfadjoint extension is of the form
\begin{equation*}
   \mD(\Bop_{\theta}) =
   \left\{ f\in \mD(\dop_0^*) \ :\
   \begin{aligned}
      &f \text{ cont. in }0,\\
      &\sqrt{2}\cos(\theta)\, f(0) = \sin(\theta) \big[ f'(+0) - f'(-0) \big]
   \end{aligned}
 \right\}
\end{equation*}
for $\theta\in [0,\pi)$.
All functions in these domains are continuous at $0$, but their derivative has a jump proportional to its value in $0$.
There is a one-to-one correspondence between the constant of proportionality and the particular selfadjoint extension.
The operators $\Bop_{heta}$ can also be interpreted as the classical harmonic oscillator with a $\delta$-interaction at $0$ on a bigger Hilbert space, see Section~\ref{sec:delta}:
\begin{align*}
   \Bop_{\theta} f
   = \left( -\frac{1}{2} \frac{\rd^2}{\rd t^2} + \frac{1}{2} t^2 + c \delta \right)f(t),
\end{align*}
with $c=\frac{\cot\theta}{\sqrt{2}}$.
Let $G(\omega)$ be as in \eqref{eq:G}.
If $c>-\frac{1}{G(0)}$, then $\Bop_{\theta}$ has only positive eigenvalues.
If $c=\frac{1}{G(0)}$, then $0$ is an eigenvalue of $\Bop_{\theta}$, and if
$c<-\frac{1}{G(0)}$, then $\Bop_{\theta}$ has exactly one negative eigenvalue.
This eigenvalue decreases monotonically to $-\infty$ as $c$ tends to $-\infty$, or equivalently, $\theta$ tends to $\pi$.
\smallskip

The free Schr\"odinger operator with singular potential at $0$ was investigated by \v{S}eba in~\cite{Seba} and then later by Kurasov in~\cite{kurasov}.
Both use von Neumann's extension theory to obtain selfadjoint extensions of a given differential operator on $\R\setminus\{0\}$ and interpret their results in terms of $\delta$- and $\delta'$-interactions at $0$.
The one-dimensional harmonic oscillator with $\delta$-interaction at the origin was considered for instance by Gadella, Glasser and Nieto in \cite{GGN} and  Viana-Gomes and Peres in \cite{VianaGomes}.
In both works the eigenfuntions are calculated in terms of confluent hypergeometric functions.
Moreover, it is shown that the eigenvalues with odd eigenfunctions are not changed, whereas the eigenvalues with even eigenfunctions increase (for $c>0$) or decrease (for $c<0$) when compared with the eigenvalues of the harmonic oscillator without singular perturbation.
\medskip

The paper is organised as follows:
In Section~\ref{sec:halfline} we consider the harmonic oscillator on the open half lines $\R_+$ and $\R_-$ and we classify all selfadjoint extensions of $\dop_\pm^{\min}$ by parametrisation with one real parameter.
In Section~\ref{sec:B} we consider the operator $\Bop$ and show that all its selfadjoint extensions are parametrised by one real parameter.
Moreover, we show that there appear arbitrarily small negative eigenvalues.
In Section~\ref{sec:C} we investigate the operator $\cop$.
The family of selfadjoint extensions is parametrised by four real parameters (equivalently by the set of all unitary $2\times 2$ matrices).
Finally, in Section~\ref{sec:delta} we give an interpretation of the operators $\Bop_{\theta}$ and $\cop_{K}$ as operators with a $\delta$- and $\delta'$-interaction at $0$ in a Hilbert space $H_- \supset L_2(\R)$.
\medskip

In this paper, we use the following notations.
We set $\R_+ = (0,\infty)$ and $\R_- := (-\infty,0)$ and for functions $f:\R\to\C$ we define their restrictions $f_\pm := f|_{\R_\pm}$.
The standard inner products on $L_2(\R)$ and on $\C^n$ are both denoted by $\langle\cdot, \cdot\rangle$. There will never be danger of confusion.

\noindent
Given a sesquilinear form $[\,\cdot\, , \,\cdot\,]$ on $\C^n$, we call a subspace $L$ \emph{neutral} if $[u,v]=0$ for every $u,v\in L$.
A subspace $L$ is called \emph{maximal neutral} if it is neutral and not properly contained in any other neutral subspace.


\section{The harmonic oscillator on the half line} 
\label{sec:halfline}

First we restrict the harmonic oscillator to the half lines $\R_\pm$.
The corresponding minimal operators are
\begin{align*}
   \dop_{\pm}^{\min} f(t) := \FA f(t)\quad (t\in\R_\pm),
   \qquad
   \mD(\dop_{\pm}^{\min}) &:= C_c^\infty(\R_\pm).
\end{align*}

These operators are in the limit point case at $\pm\infty$ and in the limit circle case at $0$, hence they are not essentially selfadjoint.
Their adjoint operators are
\begin{align*}
   \dop_{\pm}^{\max} f(t) &:= \FA f(t)
   \qquad (t\in\R_\pm),
   \\
   \mD(\dop_{\pm}^{\max}) &:=
   \{ f:\R_\pm\to\C  \ :\ f, f'\ \text{abs. cont.},\ \FA f|_{\R_\pm} \in L_2(\R_\pm) \}.
\end{align*}
Note that for $f\in\mD(\dop_+^{\max})$ the one-sided limits
$f(+0) := \lim\limits_{t\to 0^+} f(t)$ and
$f'(+0) := \lim\limits_{t\to 0^+} f'(t)$ exist.
Similarly, for $f\in\mD(\dop_-^{\max})$ the one-sided limits
$f(-0) := \lim\limits_{t\to 0^-} f(t)$ and
$f'(-0) := \lim\limits_{t\to 0^-} f'(t)$ exist.%
\smallskip

All selfadjoint extensions of $\dop_\pm^{\min}$ are given as restrictions of $\dop_\pm^{\max}$ by appropriate boundary conditions at $0$.
In Lemma~\ref{lemma:halfline} we will show that exactly one boundary condition is needed.

Recall that the \emph{defect index} of an operator $T$ with respect to $z\in\rho(T)$ is given by
\begin{equation*}
   n(T,z) := \dim(\ker (T^*-z)).
\end{equation*}
It is well known that the defect indices are constant in the complement of the \emph{numerical range}
\begin{equation*}
   W(T) := \{ \langle Tx, x \rangle : x\in\mD(T), \|x\|=1 \},
\end{equation*}
see, e.g., \cite[Ch. V, Theorem 3.2]{Kato}.
It is easy to see that $W(T)\subset\R$ for a symmetric operator $T$.
Hence its defect indices are constant in the upper and lower complex plane.
We will denote them by
\begin{equation*}
   n_+(T) = \dim(\ker (T^*-z_+)),
   \qquad
   n_-(T) = \dim(\ker (T^*-z_-)),
\end{equation*}
for any $z_\pm\in\C$ with $\im(z_\pm) \in\R_\pm$.
By the von Neumann theory, a symmetric operator has selfadjoint extensions if and only if its defect indices are equal (see for instance \cite[Ch. 8.2]{WeidmannHS}).

\begin{lemma}
   \label{lem:W(A)}
   The defect indices of $\dop_\pm^{\min}$ are
   $n_+(\dop_\pm^{\min}) = n_-(\dop_\pm^{\min}) =1$ and
   \begin{equation}
      \label{eq:W(A)}
      W( \dop_\pm^{\min} ) \subseteq \R_+.
   \end{equation}
\end{lemma}
\begin{proof}
   We show the lemma only for $\dop_+^{\min}$.
   For all $f\in\mD(\dop_+^{\min})$, integration by parts yields
   \begin{align*}
      \langle \dop_+^{\min} f, f\rangle
      &= -f'(x) \overline f(x)\Big|_0^\infty
      + \int_0^\infty |f'|^2 + x^2 |f|^2\,\rd x
      \\
      &= \int_0^\infty |f'|^2 + x^2 |f|^2\,\rd x
      > 0
   \end{align*}
   which shows \eqref{eq:W(A)}.
   Hence the defect index of $\dop_+^{\min}$ is constant in $\C\setminus\R_+$.

   It can be easily verified that two pairs of independent solutions of $(\FA + \frac{1}{2})f = 0$ are
   \begin{align}
      \label{eq:phi2}
      \phi_1(t) = \e^{\frac{1}{2} t^2},
      \qquad
      \phi_+(t) &= \e^{\frac{1}{2} t^2}\,\int_t^\infty \e^{-s^2}\,\rd s,
      \\
      \label{eq:phi3}
      \phi_1(t) = \e^{\frac{1}{2} t^2},
      \qquad
      \phi_-(t) &= \e^{\frac{1}{2} t^2}\,\int_{-\infty}^t \e^{-s^2}\,\rd s
      = \phi_+(-t).
   \end{align}
   Observe that $\phi_+ + \phi_- = \sqrt{\pi}\, \phi_1$.
   Clearly $\phi_1|_{\R_+} \notin L_2(\R_+)$,
   but $\phi_+|_{\R_+}\in L_2(\R_+)$.
   Therefore $\ker( \dop_+^{\max} + \frac{1}{2} ) = \linspan\{ \phi_+|_{\R_+} \}$
   and
   $n_+(\dop_+^{\min}) = n_-(\dop_+^{\min}) =1$.

   Analogous calculations show that
   $\langle \dop_-^{\min} f, f\rangle \ge 0$ for all $f\in\mD(\dop_-^{\min})$,
   $\phi_1|_{\R_-} \notin L_2(\R_-)$,
   $\phi_-|_{\R_-}\in L_2(\R_-)$,
   and therefore $\ker( \dop_-^{\max} + \frac{1}{2} ) = \linspan\{ \phi_-|_{\R_-} \}$
   and
   $n_+(\dop_-^{\min}) = n_-(\dop_-^{\min}) =1$.
\end{proof}

The following result on selfadjoint extensions of $\dop_\pm^{\min}$ follows easily from the general theory of Sturm-Liouville operators.
For the convenience of the reader, we present it here with a proof in order to illustrate the method of indefinite inner product spaces for the description of selfadjoint extensions.

\begin{lemma}\label{lemma:halfline}
   All selfadjoint extensions of $\dop_\pm^{\min}$ are one-dimensional;
   they are restrictions of $\dop_\pm^{\max}$ of the form
   \begin{equation}
      \mD(\dop_{\pm,\theta}) =
      \big\{ f\in\mD(\dop_\pm^{\max}) : \cos(\theta) f_\pm(+0) = \sin(\theta)f'_\pm(+0) \big\}
      \label{domain_ext}
   \end{equation}
   with $\theta\in [0, \pi)$.
\end{lemma}
\begin{proof}
   We show the claim only for $\dop_+^{\min}$ since the corresponding assertions for $\dop_-^{\min}$ follow analogously.

   \noindent
   By Lemma~\ref{lem:W(A)}, the defect index of $\dop_\pm^{\min}$ is equal to 1 on $\C\setminus\R_+$.
   Two functions $f,g\in\mD(\dop_+^{\max})$ belong to a particular selfadjoint extension of $\dop_+^{\min}$ if and only if 
   $\langle \dop_+^{\max} f, g \rangle - \langle \dop_+^{\max} f, g \rangle =0$.
   Integration by parts leads to the condition
   \begin{equation}
      \label{eq:maxsymm}
      0 = \langle \dop_+^{\max} f, g \rangle
      - \langle \dop_+^{\max} f, g \rangle
      = f(+0)\overline g'(+0) - f'(+0)\overline g(+0).
   \end{equation}
   On $\C^2$ let us define the Hermitian inner product
   \begin{equation*}
      \left[
      \begin{pmatrix} x_1\\ x_2
      \end{pmatrix},\
      \begin{pmatrix} y_1\\ y_2
      \end{pmatrix}
      \right]
      = \left\langle
      \begin{pmatrix}
         0 & -\I \\ \I & 0
      \end{pmatrix}
      \begin{pmatrix} x_1\\ x_2
      \end{pmatrix},\
      \begin{pmatrix} y_1 \\ y_2
      \end{pmatrix}
      \right\rangle
      = \I ( x_2 \overline y_1 - x_1 \overline y_2).
   \end{equation*}
   Then $f,g$ belong to a particular selfadjoint extension of $\dop_+^{\min}$ if and only if $(f(+0), f'(+0))^t$ and $(g(+0), g'(+0))^t$ belong to
   a maximal neutral subspace of $(\C^2, [\cdot,\ \cdot])$.
   Clearly $e_+=(1,\, \I)^t$ is a positive and $e_-=(-1,\, \I)^t$ is a negative vector
   and $\|e_+\|=\|e_-\|$.
   Hence all maximal neutral subspaces are given by
   \begin{align*}
      L_{\theta}
      &= \linspan\left\{
      \begin{pmatrix}
	 1-\e^{-2\I\theta}\\ \I(1+\e^{-2\I\theta})
      \end{pmatrix}
      \right\}
      = \left(
      \linspan\left\{
      \begin{pmatrix}
	 1+\e^{-2\I\theta}\\ \I(1-\e^{-2\I\theta})
      \end{pmatrix}
      \right\}
      \right)^\perp\!\!\!,
      \quad \theta\in [0,\pi).
   \end{align*}
   Therefore all selfadjoint extensions of $\dop_+^{\min}$ are given by
   \begin{align*}
      \mD(\dop_{+,\theta})
     & =
      \left\{ f\in\mD(\dop_+^{\max}) :
      \begin{pmatrix}
	 f(+0)\\ f'(+0)
      \end{pmatrix}
      \in L_{\theta}
      \right\}
      \\[1ex]
      &=
      \big\{ f\in\mD(\dop_+^{\max}) \ :\ f(+0)(1+\e^{-2\I\theta})= -\I f'(+0)(1-\e^{-2\I\theta})\big\}
   \end{align*}
   with $\theta\in [0,\pi)$.
   The last description yields (\ref{domain_ext}).
      \end{proof}

\begin{remark}
   \label{rem:symmetry}
   Let $f\in L_2(\R_+)$ and $g\in L_2(\R_-)$ such that $g(x) = f(-x)$ for $x\in\R_-$.
   From the formula \eqref{domain_ext} it is clear that
   $f\in\mD(\dop_{+,\theta})$ for some $\theta\in(0,\pi)$ if and only if
   $g\in\mD(\dop_{-,\pi-\theta})$
   and
   $f\in\mD(\dop_{+,0})$ if and only if $g\in\mD(\dop_{-,0})$.
\end{remark}
\smallskip

Recall that $\{\psi_n : n\in\N\}$, the set of eigenfunctions of the harmonic oscillator on $\R$ (see \eqref{eq:recursion}), is an orthonormal basis of $L_2(\R)$.
Let us denote
\begin{align}
   \label{eq:psipm}
   \psi_{n,+} = \psi_n|_{\R_+},
   \qquad
   \psi_{n,-} = \psi_n|_{\R_-},
   \qquad\qquad
   n\in\N_0.
\end{align}
Clearly both
$\{\psi_{2n,\pm} : n\in\N_0\}$ and $\{\psi_{2n+1,\pm} : n\in\N_0\}$
form a complete orthogonal systems on $\R_\pm$.
With these observations we can calculate the spectrum of the operators $\dop_{\pm,\theta}$ for $\theta=0$ and $\theta= \frac{\pi}{2}$.

\begin{corollary}
   Let $\dop_{\pm,\theta}$ as in \eqref{domain_ext}.
   \begin{enumerate}
      \item\label{item:i}
      $\sigma(\dop_{\pm,0}) = \sigma_p(\dop_{\pm,0})
      = \{ 2n+\frac{3}{2}  : n\in\N_0\}$
      and the corresponding eigenfunctions are $\psi_{2n+1,\pm}$, $n\in \N_0$.

      \item
      $\sigma(\dop_{\pm,\frac{\pi}{2}}) = \sigma_p(\dop_{\pm,\frac{\pi}{2}})
      = \{ 2n + \frac{1}{2} : n\in\N_0\}$
      and the corresponding eigenfunctions are $\psi_{2n,\pm}$, $n\in \N_0$.
   \end{enumerate}
\end{corollary}

\begin{proof}
   We will prove the claim only for $\dop_{+,0}$.
   All other statements are proved analogously.
   Since all $\psi_{2n+1}$ are odd functions, it follows that $\psi_{2n+1}(0)=0$ and therefore their restrictions $\psi_{2n+1,+}$ belong to $\mD(\dop_{+,0})$.
   Moreover, $\dop_{+,0}\psi_{2n+1,+} = (2n+\frac{3}{2}) \psi_{2n+1,+}$, hence
   $\{ 2n+ \frac{3}{2} : n\in\N_0\} \subset \sigma_p(\dop_{\pm,0})$.
   Now the claim follows from the completeness of the system
   $\{\psi_{2n+1,+}\}$ in $L_2(\R_+)$.
\end{proof}

From the asymptotic expansion \eqref{eq:assymptotic} of solutions of the equation
$\FA y = \lambda y$ it is clear that for every $\lambda\in\R$ there is exactly one solution $y_\lambda$ which is square integrable on $\R_+$.
It belongs to the domain of exactly one selfadjoint extension $\dop_{+,\theta}$,
namely the one with $\theta\in [0,\pi)$ such that
$\cos\theta\, y_\lambda(0) = \sin\theta\, y_\lambda'(0)$.
The expansion \eqref{eq:assymptotic} also shows that all eigenvalues are simple.
\smallskip

Since $W(\dop_\pm^{\min})\subseteq (0, \infty)$,
the non-positive spectrum
$\sigma(\dop_{\theta}) \cap (-\infty, 0]$
of any selfadjoint extension $\dop_{\theta}$
consists of at most one eigenvalue of multiplicity at most $1$, see \cite[Ch. 8.4, Cor. 2]{WeidmannHS}.
The next lemma deals with these eigenvalues.

\begin{lemma}\label{lemma:negEVA}
   Let $\dop_{\pm,\theta}$ as in \eqref{domain_ext} and let $\alpha_A$ as in \eqref{eq:delta} below.
   \begin{enumerate}

      \item
      \label{enumi:i}
      $\dop_{+,\theta}$ has the eigenvalue $0$ if and only if $\theta= \pi-\alpha_A$.
      It has a negative eigenvalue if and only if $\theta\in (\pi-\alpha_A,\,\pi)$.
      Moreover, for $j=1,2$, let $\lambda_j$ be eigenvalues of $\dop_{+,\theta_j}$.
      If $\lambda_2<\lambda_1<0$, then $\pi-\alpha_A < \theta_1 < \theta_2<\pi$.

      \item
      \label{enumi:ii}
      $\dop_{-,\theta}$ has the eigenvalue $0$ if and only if $\theta= \alpha_A$.
      It has a negative eigenvalue if and only if $\theta\in (0, \alpha_A)$.
      Moreover, for $j=1,2$, let $\lambda_j$ be eigenvalues of $\dop_{-,\theta_j}$.
      If $\lambda_2<\lambda_1<0$, then $\alpha_A > \theta_1 > \theta_2>0$.

   \end{enumerate}
\end{lemma}
\begin{proof}
   We proof only \enumiref{enumi:i}.
   The claims in \enumiref{enumi:ii} follow from \enumiref{enumi:i} and Remark~\ref{rem:symmetry}.
   Let $\omega\ge 0$.
   On $[0,\infty)$ we consider the Cauchy problem
   \begin{align}
      \label{negativeEq}
      \left( -\frac{1}{2} \frac{\rd^2}{\rd t^2} + \frac{1}{2} t^2 \right)u(t,\omega)=-\omega^2 u(t,\omega),\\
      \label{negativeBC}
      u(0,\omega)=1,\quad
      u^{\prime}_t(0,\omega)=0.
   \end{align}
   Let us write $u(t,\omega)$ as power series
   \begin{equation}
      \label{serie}
      u(t,\omega) = \sum_{n=0}^{\infty}a_n(\omega)t^n .
   \end{equation}
   Replacing \eqref{serie} in \eqref{negativeEq} we obtain $a_{2n+1}(\omega) = 0$ and
   \begin{align}
      \begin{aligned}
	 \label{a-omega}
	 a_0(\omega) &=1,\quad
	 a_2(\omega) = \omega^2
	 \\
	 a_{2n+2}(\omega) &=\frac{1}{(2n+2)(2n+1)}(2\omega^2 a_{2n}(\omega)+a_{2n-2}(\omega)),\quad n \ge 1.
      \end{aligned}
   \end{align}
   Let us show that the series converges for all $t>0$.
   To this end we will show that for every  $\omega \ge 0$ there is a constant $q(\omega)$ such that
   \begin{equation}\label{a2n}
      a_{2n}(\omega)\leq \frac{q(\omega)}{n!},
      \qquad n\in\N.
   \end{equation}
   Let us assume that this inequality is true for  numbers $n-1$ and $n$.
   Then, due to \eqref{a-omega}, we have
   \begin{align*}
      a_{2n+2}
      & \leq \frac{1}{(2n+2)(2n+1)}\left(\frac{2\omega^2 q(\omega)}{n!}+\frac{q(\omega)}{(n-1)!}\right)
      \\
      &= \frac{q(\omega)}{(n+1)!}\left( \frac{\omega^2}{2n+1}+\frac{n}{2(2n+1)}\right).
   \end{align*}
   Since $\frac{\omega^2}{2n+1}+\frac{n}{2(2n+1)}\rightarrow \frac{1}{4}$ for $n\rightarrow \infty$, there is a natural number $n_0(\omega)$ such that
   $\frac{\omega^2}{2n+1}+\frac{n}{2(2n+1)}<1$ for every $n\geq n_0(\omega)$.
   Thus, \eqref{a2n} holds if we take
   \begin{equation*}
      q(\omega)=\max \{ 1,\ 2!\cdot a_2(\omega),\ldots ,\  n_0!\cdot a_{2n_0}(\omega)\}.
   \end{equation*}
   Due to (\ref{a2n}) the series  (\ref{serie}) converges and $u(t,\omega)\le q(\omega)\,\e^{t^2}$ for $\omega \ge 0$, $t>0$.

   Note that all $a_{2n}(\omega)$ are positive increasing functions of $\omega$, so for every $t> 0$, $u(t,\omega)$ is an increasing function with respect to $\omega$ and for every $\omega\geq 0$, $u(t,\omega)$ is a positive increasing function with respect to $t$, so $u(t,\omega)\notin  L_2(\R_+)$.
   For the special case $\omega=0$ we obtain
   \begin{equation}\label{u0} u(t,0)=1+\sum_{n=1}^{\infty}\frac{t^{4n}}{\prod_{k=1}^{n}4k(4k-1)}.
   \end{equation}
   Due to the inequalities
   \begin{equation*}
      \frac{1}{3^22n(2n-1)}<\frac{1}{4n(4n-1)}<\frac{1}{2^2\, 2n(2n-1)},
      \quad n\in\N,
   \end{equation*}
   we have
   \begin{equation}\label{v0}\cosh\Big(\frac{x^2}{3}\Big)<u(t,0)<\cosh\Big(\frac{x^2}{2}\Big),\quad t>0.
   \end{equation}
   Now let us define (compare with (\ref{eq:phi2}))
   \begin{equation}\label{def_v}
      v(t,\omega) := u(t,\omega) \,\int_t^\infty \big(u(s,\omega)\big)^{-2}\,\rd s
   \end{equation}
   and
   \begin{equation}\label{eq:G}
      G(\omega) := v(0,\omega) = \int_0^\infty \big(u(s,\omega)\big)^{-2}\,\rd s .
   \end{equation}
   Since $u(t,\omega)$ is positive and increasing both in $t$ and $\omega$, it follows from \eqref{v0} that
   \begin{align*}
      v(t,\omega)
      &<u(t,\omega) \,\int_t^\infty \big(u(t,\omega)\big)^{-1} \big(u(s,\omega)\big)^{-1}\,\rd s
      =\int_t^\infty  \big(u(s,\omega)\big)^{-1}\,\rd s
      \\
      & \le \int_t^\infty \big(u(s,0)\big)^{-1}\,\rd s
      <\int_t^\infty \big(\cosh(s^2/3)\big)^{-1}\,\rd s\ \, \in  L_2(\R_+).
   \end{align*}
   It is easy to check that $v(\cdot, \omega)$ satisfies \eqref{negativeEq} and
   \begin{equation*}
      v^{\prime}_t(t,\omega)=\, -\frac{1}{u(t,\omega)}+u^{\prime}(t,\omega)\, \int_t^\infty \big(u(s,\omega)\big)^{-2}\,\rd s\, ,
   \end{equation*}
   therefore, by \eqref{negativeBC},
   $v_t'(0,\omega)=-1$.
   It follows from \eqref{eq:G} that
   \begin{equation}\label{bound_neg}
      v^{\prime}_t(0,\omega)=-\frac{v(0,\omega)}{G(\omega)}
   \end{equation}
   which is equivalent to the boundary condition \eqref{domain_ext} with
   $\theta$ such that
   $\tan\theta = - G(\omega)$, that is,
   $\theta = -\arctan (G(\omega)) \in (\pi/2, \pi)$.
   Observe that $G(\omega)$ is decreasing and continuous in $\omega$ and
   $\lim\limits_{\omega\rightarrow\infty}G(\omega)=0$.
   Hence $\theta$ is increasing in $\omega$ and tends to $\pi$ for $\omega\to\infty$.
   For the special case $\omega=0$ we obtain $\theta = \pi/2 - \alpha_A$ where
   \begin{align}
      \label{eq:delta}
      \alpha_A = \arctan(G(0)).
   \end{align}
\end{proof}


\section{One-dimensional restriction of the harmonic oscillator and classification of all selfadjoint extensions} 
\label{sec:B}

In this section we consider the harmonic oscillator on the real line with the following restriction:
\begin{equation}\label{eq:B0}
   \Bop f := \FA f,
   \qquad
   \mD(\Bop) := \{ f \in \mD(\dop) : f(0) = 0\}.
\end{equation}
The operator $\Bop$ is closed and symmetric, but not selfadjoint.
We also define the symmetric operator
\begin{align*}
   \dop_0 f := \FA f,
   \qquad
   \mD(\dop_0) := \{ f \in C_c^\infty(\R) : 0\not\in\supp f \}.
\end{align*}

\begin{remark}
   \label{remark:A0*}
   The domain of $\dop_0$ can be viewed as
   $\mD(\dop_0) =  \mD(\dop_-^{\min})\oplus\mD(\dop_+^{\min})$, hence it is easy to see that its adjoint is given by
   \begin{align*}
      \dop_0^* f(t) &=
      \begin{cases}
	 \dop_+^{\max} f_+(t),\quad & t> 0, \\
	 \dop_-^{\max} f_-(t),\quad & t< 0,
      \end{cases}
      \\
      \mD(\dop_0^*)
      &= \{ f\in L_2(\R) \,:\, f_\pm\in \mD(\dop_\pm^{\max})\}
      = \mD(\dop_-^{\max}) \oplus \mD(\dop_+^{\max}).
   \end{align*}
\end{remark}

It should be noted that $\dop_0^*$ and the selfadjoint extensions $\Bop_{\theta}$ and $\cop_{K}$ of $\Bop$ and $\cop$ which we will calculate below are not differential operators on $L_2(\R)$ in the classical sense because functions in their domains need not be continuous or differentiable in $0$.

\begin{lemma}
   We have the chain of extensions
   $\dop_0\subset\Bop\subset\Bop^*\subset\dop_0^*$ and
   \begin{align*}
      \Bop^* f(t) &= \dop_0^* f(t), \\[1ex]
      \mD(\Bop^*) &= \{ f:\R\to\C \,:\, f_\pm\in \mD(\dop_\pm^{\max}),\ f(-0) = f(+0)\}.
   \end{align*}
\end{lemma}

\begin{proof}
   Note that $\dop_0\subset\Bop$, hence $\Bop^*\subset\dop_0^*$.
   Let $f\in\mD(\Bop)$ and $g\in \mD(\dop_0^*)$
   Then $f$ and $f'$ are continuous in $0$, $f(0)=0$ and $g_\pm\in\mD(\dop_\pm^{\max})$.
   Hence integration by parts yields
   \begin{multline*}
      \langle \Bop f, g \rangle - \langle f, \dop_0^* g \rangle
      \\
      \begin{aligned}
	 &=
	 -\int_{-\infty}^0 f'' \overline g\, \rd t
	 - \int_0^{\infty} f'' \overline g\, \rd t
	 + \int_{-\infty}^0 f \overline g''\, \rd t
	 + \int_0^{\infty} f \overline g''\, \rd t
	 \\
	 &=
	 -f'(t) \overline g(t) \Big|_{-\infty}^0
	 -f'(t) \overline g(t) \Big|_0^\infty
	 +f(t) \overline g'(t) \Big|_{-\infty}^0
	 +f(t) \overline g'(t) \Big|_0^\infty
	 \\
	 &=
	 -f'(0) \overline g(-0)
	 +f'(0) \overline g(+0)
	 \\
	 &=
	 f'(0) \big( \overline g(+0) - \overline g(-0) \big).
      \end{aligned}
   \end{multline*}
   Therefore $g\in\mD(\Bop^*)$ if and only if $g$ is continuous in $0$ and in this case $\Bop^* g = \dop_0^* g$.
\end{proof}

Note that functions in the domain of $\Bop^*$ are continuous but their derivative may have a discontinuity in $0$.
\smallskip

Analogously to Lemma~\ref{lemma:halfline} we now classify all selfadjoint extensions of $\Bop$.
\begin{proposition}
   \label{prop:B0}
   The defect indices of $\Bop$ are
   $n_+(\Bop) = n_-(\Bop) = 1$.
   Hence all selfadjoint extensions of $\Bop$ are one-dimensional restrictions of $\Bop^*$.
   They are of the form
   \begin{equation*}
      \mD(\Bop_{\theta}) =
      \big\{ f\in\mD(\Bop^{*}) :
      \sqrt{2}\cos\theta\, f(0) = \sin\theta\, \big[ f'(+0) - f'(-0) \big] \big\}
   \end{equation*}
   for $\theta\in [0,\pi)$.
\end{proposition}
\begin{proof}
   Observe that $\Bop$ is a restriction of $\dop$, hence for the numerical ranges we have the inclusion $W(\Bop)\subset W(\dop) \subset [0,\infty)$,
   so the defect index of $\Bop$ is constant in $\C\setminus [0,\infty)$
   and it suffices to show that $\dim(\ker(\Bop^*+\frac{1}{2})) = 1$.

   From the proof of Lemma~\ref{lem:W(A)} it is clear that every $L_2$-solution of $(\FA+\frac{1}{2})f=0$ must be of the form $f=\alpha_+ \phi_+ \chi_{(-\infty, 0)} + \alpha_- \phi_- \chi_{(0,\infty)}$ with $\alpha_\pm\in\C$ and $\phi_\pm$ as in \eqref{eq:phi2} and \eqref{eq:phi3}.
   Clearly every such function belongs to $\mD(\dop_0^*)$ and
   $f(-0) = \frac{\alpha_-}{2} \sqrt{\pi}$ and
   $f(-0) = \frac{\alpha_+}{2} \sqrt{\pi}$.
   For $f\in\mD(\Bop^*)$ must have that $f(-0) = f(+0)$.
   Therefore
   \begin{align*}
      \ker\Big(\Bop^*+\frac{1}{2}\Big)
      = \linspan\{ \phi_+ \chi_{(-\infty, 0)} + \phi_- \chi_{(0,\infty)} \}
   \end{align*}
   which shows that $n(\Bop, -\frac{1}{2}) =1$.

   Let us now determine all selfadjoint extensions of $\Bop$.
   This is equivalent to determine all selfadjoint restrictions of $\Bop^*$.
   Let $f,g\in\mD(\Bop^*)$.
   Performing integration by parts we find
   \begin{multline*}
      \langle f, \Bop^* g \rangle - \langle \Bop^* f, g \rangle
      \\
      \begin{aligned}
	 &=
	 \int_{-\infty}^0 f'' \overline g \rd t
	 + \int_0^{\infty} f'' \overline g \rd t
	 - \int_{-\infty}^0 f \overline g'' \rd t
	 - \int_0^{\infty} f \overline g'' \rd t
	 \\
	 &=
	 f'(t) \overline g(t) \Big|_{-\infty}^0
	 +f'(t) \overline g(t) \Big|_0^\infty
	 -f(t) \overline g'(t) \Big|_{-\infty}^0
	 -f(t) \overline g'(t) \Big|_0^\infty
	 \\
	 &=
	 f'(-0) \overline g(-0)
	 -f'(+0) \overline g(+0)
	 -f(-0) \overline g'(-0)
	 +f(+0) \overline g'(+0)
	 \\
	 &=
	 \big[ f'(-0) - f'(+0) \big] \overline g(0)
	 - \big[  \overline g'(-0) -\overline g'(+0) \big] f(0).
      \end{aligned}
   \end{multline*}
   Set
   \begin{equation*}
      G=
      \begin{pmatrix}
	 0   & -\I & \I \\
	 \I  & 0   & 0  \\
	 -\I & 0   & 0  \\
      \end{pmatrix}.
   \end{equation*}
   Then $f,g$ belong to a particular selfadjoint extension of $\Bop$ if and only if
   $(f(0),\, f'(-0),\, f'(+0))^t$ and
   $(g(0),\, g'(-0),\, g'(+0))^t$ belong to a maximal neutral subspace of $(\C^3, [\cdot,\,\cdot])$ where
   \begin{align*}
      \left[
      \begin{pmatrix} x_1\\ x_2\\ x_3
      \end{pmatrix},\
      \begin{pmatrix} y_1\\ y_2\\ y_3
      \end{pmatrix}
      \right]
      &:=
      \left\langle
      G
      \begin{pmatrix} x_1\\ x_2\\ x_3
      \end{pmatrix}
      ,\
      \begin{pmatrix} y_1\\ y_2\\ y_3
      \end{pmatrix}
      \right\rangle
      =
      \I \left[
      x_1 \big( \overline y_3 - \overline y_2 \big)
      +
      \overline y_1 \big( x_2 - x_3 \big)
      \right].
   \end{align*}
   The eigenvalues of $G$ are $0$ and $\pm\sqrt{2}$ with eigenspaces
   \begin{align*}
      \ker G &= \linspan\left\{ (0,\, 1,\, 1)^t \right\},
      \\
      L_+ := \ker (G -\sqrt{2}) &= \linspan\left\{ (\I\sqrt{2},\, -1,\, 1)^t \right\},
      \\
      L_- := \ker (G +\sqrt{2}) &= \linspan\left\{ (-\I\sqrt{2},\, -1,\, 1)^t \right\}.
   \end{align*}
   Note that $L_\pm$ are maximal positive and maximal negative subspaces of $\C^3$ respectively.
   Hence any maximal neutral subspace of $\C^3$ has dimension 2.
   They are of the form $\ker G \oplus \{ v + Kv: v\in \ker(G-\sqrt{2}) \}$ where $K$ is an isometry from $L_+$ to $L_-$.
   Clearly all such isometries are of the form $v_+\mapsto \e^{-2\I\theta} v_-$ where $v_\pm\in L_\pm$.
   In summary, all maximal neutral subspaces are
   \begin{equation*}
      L_\theta
      = \linspan \left\{
      \begin{pmatrix} 0 \\ 1 \\ 1
      \end{pmatrix},\
      \begin{pmatrix} \I\sqrt{2} \\ -1 \\ 1
      \end{pmatrix}
      +\e^{2\I\theta}
      \begin{pmatrix} -\I\sqrt{2} \\ -1 \\ 1
      \end{pmatrix}
      \right\},
      \qquad
      \theta\in [0,\pi).
   \end{equation*}
   This can be rewritten as
   \begin{align*}
      L_\theta
      &= \linspan \left\{
      \begin{pmatrix} 0 \\ 1 \\ 1
      \end{pmatrix},\
      \begin{pmatrix} \sqrt{2} \sin(\theta)\\ -\cos(\theta) \\ \cos(\theta)
      \end{pmatrix}
      \right\}
      =
      \left(
      \linspan\left\{
      \begin{pmatrix} \sqrt{2} \cos(\theta)\\ \sin(\theta) \\ -\sin(\theta)
      \end{pmatrix}
      \right\}
     \right) ^\perp.
   \end{align*}
   It follows that $f\in\mD(\Bop_{\theta})$ if and only if $f\in\mD(\Bop^*)$ and
   \begin{equation}
      \label{eq:BopBC}
      \sqrt{2} \cos\theta\, f(0) + \sin\theta\, \big[ f'(-0) - f'(+0) \big] = 0.
   \end{equation}
\end{proof}

\begin{remark}
   \begin{enumerate}
      \item
      Observe that functions $f$ in $\mD(\Bop)^*$ which satisfy
      \\
      $\{(f(0), f'(-0), f'(+0)\}\in \ker G$ belong to $\mD(\Bop)$.

      \item
      The selfadjoint extensions of $\Bop$ can be divided into the following cases:
      \begin{enumerate}
	 \item $\theta = \frac{\pi}{2}$.
	 In this case the boundary condition \eqref{eq:BopBC} simplifies to
	 \begin{equation*}
	    f'(-0) - f'(+0) = 0.
	 \end{equation*}
	 That is, $f$ and $f'$ are continuous and we obtain the classical harmonic oscillator: $\Bop_{\pi/2} = \dop$.

	 \item $\theta = 0$.
	 In this case the boundary condition \eqref{eq:BopBC} simplifies to
	 $f(0) = 0.$

	 \item $\theta \in (0, \pi)\setminus\{\pi/2\}$.
	 The boundary condition \eqref{eq:BopBC} can be written as
	 \begin{equation*}
	    \label{eq:BopBCsimplified}
	    f(0)
	    =
	    \frac{ \tan(\theta) }{\sqrt{2} }
	    \big[ f'(+0) - f'(-0) \big].
	 \end{equation*}
	 Hence any function in $\mD(\Bop_{\theta})$ is continuous but its derivative has a jump in $t=0$ which is proportional to the value of $f$ in $0$.
	 Two different selfadjoint extensions of $\Bop$ have different constants of proportionality.
      \end{enumerate}

      \item
      For every $n\in\N_0$, the function $\psi_{2n+1}$ from \eqref{eq:recursion} is an eigenfunction of $\Bop_{\theta}$ with eigenvalue $2n+3/2$.
      So the odd eigenvalues of the harmonic oscillator are not affected by the boundary condition at $0$.
   \end{enumerate}
\end{remark}

An interpretation of these operators as a differential operator with a $\delta$-interaction at $0$ is given in Section~\ref{sec:delta}.
\smallskip

Let $\lambda\in\C$.
By the asymptotic expansion \eqref{eq:assymptotic} the equation $\FA y = \lambda y$ has square integrable solutions $y_\pm$ on $\R_\pm$ which are unique up to a constant factor.
Let us define
\begin{align}\label{eq:y}
   y(x) =
   \begin{cases}
      y_+(x), & x\ge0,\\
      y_+(-x), & x<0.
   \end{cases}
\end{align}
Then clearly $y\in\mD(\Bop^*)$ and it is, up to a constant factor, the unique solution of $(\Bop^* -\lambda)y =0$.
Moreover, $y\in \mD(\Bop_{\theta})$ where $\theta$ is the unique number in $[0,\pi)$ such that $\sqrt{2}\cos\theta\, f(0) = \sin\theta\, [ f'(+0) - f'(-0)]$.
\smallskip

This shows that, as in the case of selfadjoint extensions of $\dop_0$, every $\lambda\in\R$ appears as eigenvalue of exactly one selfadjoint extension of $\Bop$ and that every eigenvalue is simple.
Moreover, any given $\Bop_{\theta}$ can have at most one negative eigenvalue.

As in Lemma~\ref{lemma:negEVA} we can identify all $\theta$ such that $\Bop_{\theta}$ has a negative eigenvalue.
\begin{lemma}\label{lemma:negEVB}
   Let $\Bop_{\theta}$ as in Proposition~\ref{prop:B0} and let $\alpha_B = \arctan(\frac{G(0)}{\sqrt{2}})$ with $G$ as in \eqref{eq:G}.
   Then $\Bop_{\theta}$ has the eigenvalue $0$ if and only if $\theta= \pi-\alpha_B$.
   It has a negative eigenvalue if and only if $\theta\in (\pi-\alpha_B,\,\pi)$.
   Moreover, for $j=1,2$, let $\lambda_j$ be eigenvalues of $\Bop_{\theta_j}$.
   If $\lambda_2<\lambda_1<0$, then $\pi-\alpha_B < \theta_1 < \theta_2<\pi$.
\end{lemma}
\begin{proof}
Let $\lambda = -\omega^2<0$ and $y$ as in \eqref{eq:y} with $y_+ = v(\cdot, \omega)$ (cf. \eqref{def_v}).
Then $y$ is an eigenfunction of $\Bop_{\theta}$ with eigenvalue $\lambda$ if and only if $\tan(\theta) = -\frac{G(\omega)}{\sqrt{2}}$.
Hence negative eigenvalues occur if and only if
$\theta \in (\pi - \arctan(\frac{G(0)}{\sqrt{2}}),\, \pi)$.
Since $G$ is decreasing in $\omega$ with $\lim_{\omega\to\infty}G(\omega) = 0$, also the last claim follows.
\end{proof}

For $\lambda\in -(2\N-\frac{1}{2})$ we can calculate the corresponding eigenfunctions by a recursion formula.

\begin{lemma}
   Let $n\in\N$ and $\phi_\pm$ as in \eqref{eq:phi2} and \eqref{eq:phi3}.
   Set $\phi_{\pm,n}(t) := \left( \frac{\rd}{\rd t} + t \right)^n \phi_\pm(t)$ for $t\in\R_\pm$ and
   \begin{equation*}
      u_n(t) :=
      \begin{cases}
	 \phi_{+,n}(t), & t > 0,\\[2ex]
	 \phi_{-,n}(t), & t < 0.
      \end{cases}
   \end{equation*}
   Then $u_{2n}\in\mD(\Bop^*)$.
   That is, $u_{2n}$ defines a selfadjoint extension $\Bop_{\theta}$ of $\Bop$ and it is an eigenfunction of $\Bop_{\theta}$ with eigenvalue $-2n-\frac{1}{2}$.
\end{lemma}

\begin{proof}
   Clearly $u_n\in\mD(\dop_0^{\max})$.
   It is easy to check that
   $\phi_+(t) = \phi_-(-t)$ for $t>0$.
   Moreover, a straightforward calculation shows
   \begin{align*}
      \left(\FA+n+\frac{1}{2} \right)\phi_{\pm,n} = 0.
   \end{align*}
   Hence
   $\phi_{+,n}(t) = (-1)^n \phi_{-,n}(-t)$ for $t>0$.
   So $u_{2n}$ is continuous in $0$, and therefore it belongs to $\mD(\Bop^*)$.
   If in addition we had $u'_{2n}(-0) = u'_{2n}(+0)$, then $u_{2n}\in\mD(\dop)$, and $-2n-\frac{1}{2}$ would be an eigenvalue of $\dop$, in contradiction to \eqref{eq:Aspectrum}.
\end{proof}

\section{Two-dimensional restriction of the harmonic oscillator and classification of its selfadjoint extensions} 
\label{sec:C}

Let us restrict the harmonic oscillator on the real line further.
We consider the following restriction $\cop$ of the selfadjoint operator $\dop$:
\begin{align*}
   \cop f := \dop f,
   \qquad
   \mD(\cop) := \{ f \in \mD(\dop) : f(0) = f'(0) = 0\}.
\end{align*}
The operator $\cop$ is closed and symmetric, but not selfadjoint.
It is easy to see
\begin{align*}
   \cop^* = \dop_0^*.
\end{align*}
The operator $\cop$ is closely related to the harmonic oscillator on the half lines $\R_\pm$ because
\begin{equation}
   \label{eq:CA}
   \cop = (\cop^*)^* = (\dop_0^*)^* =
   \overline{\dop_0} = \overline{\dop_-^{\min}}\oplus \overline{\dop_+^{\min}}.
\end{equation}

Analogously to Lemma~\ref{lemma:halfline} and Proposition~\ref{prop:B0} we now classify all selfadjoint extensions of $\cop$.
Observe that the selfadjoint extensions of $\cop$ are exactly those of $\dop_0$.
\smallskip

Recall that $U(2)$ is the set of all unitary $2\times 2$ matrices.
\begin{proposition}
   The defect indices of $\cop$ are
   $n_+(\cop) = n_-(\cop) =2$.
   Hence all selfadjoint extensions of $\cop$ are two-dimensional restrictions of $\cop^*$.
   There is a bijection from $U(2)$ to the set of all selfadjoint extensions of $\cop$ given as follows:
   For every $K = (k_{jk})_{j,k=1}^2 \in U(2)$, the operator
   \begin{align}
      \nonumber
      \cop_{K} f &= \FA f,
      \\[2ex]
      \label{eq:Kbc}
      \mD(\cop_{K}) &=
      \left\{ f\in\mD(\cop^*) \ :\
      \begin{aligned}
	 0 &=
	 ( 1 - k_{11}) f(-0)
	 + \I(1 + k_{11} ) f'(-0)
	 \\
	 &\phantom{=}+ \I k_{12} f(+0)
	 - k_{12} f'(+0),
	 \\[1ex]
	 0 &=
	 -k_{21} f(-0)
	 + \I k_{21} f'(-0)
	 \\
	 &\phantom{=}+\I(1+k_{22}) f(+0)
	 + (1-k_{22}) f'(+0)
      \end{aligned}
      \right \}
   \end{align}
   is a selfadjoint extension of $\cop$.
   There are no other selfadjoint extensions and $\cop_K=\cop_{\widetilde K}$ if and only if $K=\widetilde K$.
\end{proposition}
For a parametrisation of the selfadjoint extensions with four real parameters, see the corollary after the proof of this proposition.

\begin{proof}
   From Lemma~\ref{lemma:halfline} we know that
   $n_+(\dop_\pm^{\min}) = n_-(\dop_\pm^{\min}) = 1$, so
   \begin{align*}
      \dim(\ker(\dop_\pm^{\max} - \I)) = \dim(\ker(\dop_\pm^{\max} + \I)) = 1.
   \end{align*}
   Hence there are functions $\psi_\pm\neq 0$ such that
   $\ker(\dop_\pm^{\max} -\I) = \linspan\{ \psi_{\pm} \}$.
   From Remark~\ref{remark:A0*} it is clear that $\eta\in\ker(\dop_0^*-\I)$ if and only if $\eta|_{\R_\pm} \in \ker(\dop_\pm^{\max} -\I)$.
   Therefore
   \begin{align*}
      \ker(\dop_0^* -\I) = \linspan\{ \chi_{\R_-}\psi_-,\  \chi_{\R_+}\psi_+ \}
   \end{align*}
   and $n_+(\dop_0^*) = 2$.
   Analogously $n_-(\dop_0^*) = 2$ can be shown.

   Now let us determine all selfadjoint extensions of $\dop_0$ which is equivalent to determine all selfadjoint restrictions of $\dop_0^*$.
   Again we apply integration by parts and find
   \begin{multline*}
      \langle f, \dop_0^* g\rangle
      - \langle \dop_0^* f, g\rangle
      \\
      =
      f(+0)\overline g'(+0) - f'(+0)\overline g(+0)
      -f(-0)\overline g'(-0) + f'(-0)\overline g(-0)
   \end{multline*}
   for all $f,g\in\mD(\dop_0^*)$.

   Hence $f,g$ belong to a particular selfadjoint extension of $\dop_0$ if and only if
   $(f(-0),\, f'(-0),\, f(+0),\, f'(+0))^t$ and
   $(g(-0),\, g'(-0),\, g(+0),\, g'(+0))^t$ belong to a maximal neutral subspace of $(\C^4, [\cdot,\,\cdot])$ where
   \begin{align*}
      \left[
      \begin{pmatrix} x_1\\ x_2\\ x_3\\ x_4
      \end{pmatrix},\
      \begin{pmatrix} y_1\\ y_2\\ y_3\\ y_4
      \end{pmatrix}
      \right]
      &:=
      \left\langle
      \begin{pmatrix}
	 0 & -\I & 0 & 0 \\
	 \I & 0 & 0 & 0\\
	 0 & 0 & 0 & \I\\
	 0 & 0 & -\I & 0
      \end{pmatrix}
      \begin{pmatrix} x_1\\ x_2\\ x_3\\ x_4
      \end{pmatrix}
      ,\
      \begin{pmatrix} y_1\\ y_2\\ y_3\\ y_4
      \end{pmatrix}
      \right\rangle
      \\[1ex]
      &=
      \I( x_1\overline y_2 - x_2\overline y_1
      -x_3\overline y_4 + x_4\overline y_3).
   \end{align*}
   Every maximal neutral subspace has dimension 2.
   Let
   \begin{alignat*}{2}
      v_1 &= \frac{1}{\sqrt{2}}(1,\,\I,\, 0,\,0)^t,\
      &v_2 &= \frac{1}{\sqrt{2}}(0,\, 0,\, \I,\,1)^t,
      \\
      w_1 &= \frac{1}{\sqrt{2}}(1,\,-\I,\, 0,\,0)^t,\
      &w_2 &= \frac{1}{\sqrt{2}}(0,\, 0,\, -\I,\,1)^t.
   \end{alignat*}
   Then $L_+ = \linspan\{ v_1,\, v_2\}$ is a maximal positive and
   $L_- = \linspan\{ w_1,\, w_2\}$ is a maximal negative subspace of $(\C^4,\ [\cdot,\cdot])$ and all maximal neutral subspaces are of the form
   \begin{align*}
      L_K = \{ v + Kv : v\in L_+ \}
      &= \{ w + K^*w : w\in L_- \}^{[\perp]}
      \\
      &= \{ w - K^*w : w\in L_- \}^\perp
   \end{align*}
   where $K$ is a unitary operator from $L_+$ to $L_-$ and $[\perp]$ denotes the orthogonal complement with respect to the inner product $[\,\cdot\,, \,\cdot\,]$.
   With respect to the basis vectors $v_1, v_2, w_1, w_2$, $K$ can be written as quadratic matrix
   \begin{equation}
      \label{eq:Kmatrix}
      K =
      \begin{pmatrix}
	 k_{11} & k_{12} \\
	 k_{21} & k_{22}
      \end{pmatrix}
   \end{equation}
   with $k_{jk}\in \C$ (for the form of these numbers see the corollary after this proof).
   With respect to the standard unit vectors $\e_1, \e_2, \e_3, \e_4$ in $\C^4$, the space $L_K$ can be written as
   \begin{align}
      \nonumber
      L_K &=
      \linspan\left\{
      \begin{pmatrix}
	 1+k_{11}\\
	 \I(1-k_{11})\\
	 -\I k_{21}\\
	 k_{21}
      \end{pmatrix},\
      \begin{pmatrix}
	 k_{12}\\
	 -\I k_{12}\\
	 \I(1-k_{22})\\
	 1+k_{22}
      \end{pmatrix}
      \right\}
      \\[2ex]
      &=
      \label{eq:orthogonal}
      \left(
      \linspan\left\{
      \begin{pmatrix}
	 1 - \overline k_{11}     \\
	 -\I(1+ \overline k_{11})\\
	 -\I \overline k_{12}     \\
	 -\overline k_{12}
      \end{pmatrix},\
      \begin{pmatrix}
	 -\overline k_{21}\\
	 -\I \overline k_{21}\\
	 -\I(1+\overline k_{22})\\
	 1-\overline k_{22}
      \end{pmatrix}
      \right\}
      \right)^\perp
   \end{align}
   where $K=(k_{ij})_{ij=1}^2$ as in \eqref{eq:Kmatrix}.
   From \eqref{eq:orthogonal} it follows that every selfadjoint extension of $\cop$ is of the form \eqref{eq:Kbc}.
\end{proof}

It is well-known that $U(2)$ is parametrised by four real parameters $\phi, \alpha, \beta_1, \beta_2\in\R$: Every $K\in U(2)$ is of the form
\begin{equation}
   \label{eq:K}
   K =
   \e^{\I\phi}
   \begin{pmatrix}
      \e^{\I\beta_1}\sin\alpha & \e^{-\I\beta_2}\cos\alpha\\
      \e^{\I\beta_2}\cos\alpha & -\e^{-\I\beta_1}\sin\alpha
   \end{pmatrix}
\end{equation}
for fixed $\phi, \alpha, \beta_1, \beta_2$.

Therefore the boundary conditions in \eqref{eq:Kbc} can be rewritten as follows:

\begin{corollary}
   Let $K\in U(2)$ as in \eqref{eq:K}.
   Then $f\in \mD(\cop_{K})$ if and only if $f\in \mD(\cop^*)$ and $f$ satisfies the boundary conditions
   \begin{align}
      \begin{aligned}
	 \label{eq:realbc}
	 0 &=
	 \begin{alignedat}[t]{4}
	    &  & (1 - \e^{\I\phi}\e^{\I\beta_1}\sin\alpha )\ &f(-0)
	    &&\ +\ & \I(1+ \e^{\I\phi}\e^{\I\beta_1}\sin\alpha )\ &f'(-0)
	    \\
	    &\phantom{=}+& \I \e^{\I\phi}\e^{-\I\beta_2}\cos\alpha\ &f(+0)
	    &&\ -\ & \e^{\I\phi}\e^{-\I\beta_2}\cos\alpha\ &f'(+0),
	 \end{alignedat}
	 \\[2ex]
	 0 &=
	 \begin{alignedat}[t]{4}
	    && -\e^{\I\phi}\e^{\I\beta_2}\cos\alpha\ &f(-0)
	    &&\ +\ & \I \e^{\I\phi}\e^{\I\beta_2}\cos\alpha\ &f'(-0)
	    \\
	    &\phantom{=}\ +& \I(1 - \e^{\I\phi}\e^{-\I\beta_1}\sin\alpha )\ & f(+0)
	    &&\ +\ & (1+ \e^{\I\phi}\e^{-\I\beta_1}\sin\alpha )\ & f'(+0).
	 \end{alignedat}
      \end{aligned}
   \end{align}
\end{corollary}
\medskip

In the following subsections we discuss particular choices of $K$.

\subsection{Classical harmonic oscillator} 
\label{subsec:C:classical}
Let
$K = \begin{pmatrix} 0 & -1 \\ -1 & 0
\end{pmatrix}$.
For instance, we can choose $\alpha = \beta_1=\beta_2=0$, $\phi = \pi$.
Then the boundary conditions \eqref{eq:Kbc}
reduce to
\begin{align*}
   f(-0) = f(+0) \quad\text{and}\quad f'(-0) = f'(+0).
\end{align*}
Hence $\cop_{K}=\dop$ is the classical harmonic oscillator.

\subsection{Boundary conditions such that $\cop_{K}=\Bop_{\theta}$} 
\label{subsec:C:delta}

Let
$\beta_1 = \beta_2 = 0$,
$\alpha\in(0,\pi)$,
$\phi = \alpha+\pi/2$.
Then
\begin{equation*}
K = \I\e^{\I\alpha}
\begin{pmatrix} \sin\alpha & \cos\alpha \\ \cos\alpha & -\sin\alpha
\end{pmatrix}.
\end{equation*}
and the boundary conditions \eqref{eq:realbc} become
\begin{align}
   \label{eq:deltaBC}
   \begin{aligned}
   0 &=
   \left( 1 - \I\e^{\I\alpha}\sin\alpha \right) f(-0)
   + \I \left(1 + \I\e^{\I\alpha}\sin\alpha \right) f'(-0)
   \\
   &\phantom{=} - \e^{\I\alpha}\cos\alpha \Big( f(+0) + \I f'(+0) \Big),
   \\[1ex]
   0 &=
   -\I\e^{\I\alpha}\cos\alpha \Big( f(-0) - \I f'(-0) \Big)
   \\
   &\phantom{=}+\I \left( 1 - \I\e^{\I\alpha}\sin\alpha \right) f(+0)
   +  \left( 1 + \I\e^{\I\alpha}\sin\alpha \right) f'(+0),
   \end{aligned}
\end{align}
which, for $\alpha\neq \pi/2$ is true if and only if
\begin{equation}
   \label{eq:BCdelta}
   f(-0) = f(+0) =: f(0)
   \quad\text{and}\quad f'(-0) - f'(+0) = 2\tan\alpha\, f(0).
\end{equation}
Choose $\theta\in[0,\pi)$ such that $\cot\theta = -\sqrt{2} \tan\alpha$.
Then $\cop_{K}=\Bop_{\theta}$ with $K$ as above.
For $\alpha = \pi/2$, the conditions \eqref{eq:deltaBC} are equivalent to
$f(+0)=f(-0) = 0$.


\subsection{Boundary conditions with continuous derivative} 
\label{subsec:C:deltaprime}

Let $\alpha\in (0, \pi)$ and let $\beta_1=\beta_2=0$, $\phi=\pi/2 - \alpha$.
Note that $\e^{\I\phi}=\I\e^{-\I\alpha}$ and
\begin{equation*}
K = \I\e^{-\I\alpha}
\begin{pmatrix} \sin\alpha & \cos\alpha \\ \cos\alpha & -\sin\alpha
\end{pmatrix}.
\end{equation*}
Then the equations (\ref{eq:realbc}) become
\begin{align}
   \begin{aligned}
      \label{eq:deltaprime}
      0  & =  \e^{-\I\alpha} \cos\alpha\ f(-0)
      + \I(1+ \I\e^{-\I\alpha}\sin\alpha)\ f'(-0)
      \\
      &\phantom{=\ } -\e^{-\I\alpha} \cos\alpha  \Big(f(+0) + \I f'(+0) \Big)
      \\[2ex]
      0 & =
      -\e^{-\I\alpha}\cos\alpha \Big(\I f(-0) + f'(-0) \Big)
      \\
      &\phantom{=\ } + \I\e^{-\I\alpha}\cos\alpha\ f(+0)
      + (1+ \I\e^{-\I\alpha}\sin\alpha) f'(+0).
   \end{aligned}
\end{align}
If $\alpha\neq \pi/2$, then \eqref{eq:deltaprime} is equivalent to
\begin{equation}
   \label{eq:bcdeltaprime}
   f^{\prime}(-0)=f^{\prime }(+0)=:f^{\prime}(0)
   \quad\text{and}\quad
   f(+0)-f(-0)= -2\tan\alpha\, f^{\prime}(0).
\end{equation}
If $\alpha = \pi/2$, then \eqref{eq:deltaprime} is equivalent to
$f^{\prime}(-0)=f^{\prime }(+0)= 0$.



\section{Interpretation as $\delta$- and $\delta'$-potentials} 
\label{sec:delta}

Observe that the operator $\dop$ from \eqref{eq:A} is closed. Hence the set $H_+:= \mD(\dop)$ becomes a Hilbert space with the norm
\begin{align*}
   \|f\|_+ := \|f\|_\dop := \left( \|f\|^2 + \|\dop f\|^2 \right)^{\frac{1}{2}},
   \qquad f\in H_+.
\end{align*}
Let $H_0:= L_2(\R)$.
In addition to the usual norm on $H$, we define
\begin{equation*}
   \|f\|_- := \sup\{ | \langle f,g\rangle | : g\in H_+, \|g\|_+\le 1 \},
   \qquad f\in H_0,
\end{equation*}
and we define $H_-$ to be the closure of $H_0$ with respect to the norm $\|\cdot\|_-$.
Then $(H_-, \|\cdot\|_-)$ is a Hilbert space and it can be viewed as the dual space of $H_+$.
Observe that we have the continuous inclusions
\begin{equation*}
   H_+ \subset H_0 \subset H_-.
\end{equation*}
On says that $H_0$ is \emph{rigged} by $H_+$ and $H_-$, see, for instance, \cite{UsShB2}, Chapter 14.

If $T:H_+\to H_0$ is a bounded linear operator, then we define its adjoint operator $T^*:H_0\to H_-$ as the unique bounded linear operator that satisfies
\begin{equation*}
   \langle Tf, g\rangle = \langle f, T^*g\rangle,
   \qquad f\in H_+,\, g\in H_0,
\end{equation*}
where $\langle\cdot\,,\,\cdot\rangle$ denotes the inner product on $H_0$.

Let us define the functions
\begin{align*}
   w_1(t)= \begin{cases}
      v(t,0), & \text{ if $t>0$,}\\
      v(-t,0),& \text{ if $t<0$,}
   \end{cases}
   \qquad
   w_2(t)= \begin{cases}
      v(t,0), & \text{ if $t>0$,}\\
      -v(-t,0),& \text{ if $t<0$,}
   \end{cases}
\end{align*}
with $v$ as in \eqref{def_v}.
Clearly $w_1, w_2\in H_0\subset H_-$.
Observe that
\begin{align}
\label{eq:w}
&w_1(+0) = w_1(-0) = w_2(+0) = -w_2(-0)=v(0,0) = G(0)
\\
\text{and}\quad
&w_1'(+0) = -w_1'(-0) = w_2'(+0) = w_2'(-0)=v'(0,0) = -1.
\end{align}

\begin{lemma}\label{lem:delta}
   The linear functionals
   \begin{align*}
      \delta &: H_+\to\C,\quad \delta f = f(0),\\
      \delta'&: H_+\to\C,\quad \delta' f = f'(0)
   \end{align*}
   are bounded and
   \begin{align*}
      \delta f = \frac{1}{2} \langle \dop f, w_1\rangle,
      \quad
      \delta' f = \frac{1}{2G(0)} \langle \dop f, w_2\rangle,
      \qquad
      f\in H_+.
   \end{align*}
\end{lemma}

\begin{proof}
   Note that for any $f\in H_+=\mD(\dop)$ and $j=1,2$, we have, using integration by parts twice,
   \begin{align*}
      \langle \dop f, w_j\rangle
      &= \int_{-\infty}^{+\infty}(\dop f)(t)\cdot w_j(t) \,\rd t
      \\
      &= \int_{-\infty}^0(-f^{\prime\prime}(t)+t^2f(t))w_j(t)\, \rd t\
      +
      \int_0^{+\infty}(-f^{\prime\prime}(t)+t^2f(t))w_j(t)\, \rd t\
      \\
      &= \int_{-\infty}^0f(t)(-w_j^{\prime\prime}(t)+t^2w_j(t))\, \rd t\
      +
      \int_0^{+\infty} f(t) (-w_j^{\prime\prime}(t)+t^2w_j(t))\, \rd t\
      \\
      &\phantom{=} + f^{\prime}(0)\{ w_j(+0)-w_j(-0)\} + f(0)\{w_j^{\prime}(-0)-w_j^{\prime}(+0)\}
      \\[1ex]
      &=
      f^{\prime}(0)\{ w_j(+0)-w_j(-0)\} +  f(0)\{w_j^{\prime}(-0)-w_j^{\prime}(+0)\},
   \end{align*}
   so the second claim follows from \eqref{eq:w}.
   Now the boundedness of $\delta$ and $\delta'$ is clear, because
   $|\delta f| = |\frac{1}{2} \langle \dop f, w_1 \rangle|
   \le \frac{1}{2} \| \dop f \|\, \|w_1\|
   \le \frac{1}{2} \| f \|_+ \|w_1\|$,
   and analogously
   $|\delta' f| \le \frac{1}{2G(0)} \| f \|_+ \|w_1\|$.
\end{proof}

Recall that in our case, $H_+\subset \mD(\dop_0^*) \subset H_0 \subset H_-$.
By definition of $H_+$, the operator
\begin{align*}
   \widetilde \dop: H_+\to H_0,\qquad
   \widetilde \dop f = \dop f
\end{align*}
is bounded.
Let us calculate how $\widetilde \dop^*$ acts on elements $g\in \mD(\dop_0^*)$.
As in the proof of Lemma~\ref{lem:delta}, integration by parts gives for $f\in H_+$
\begin{align*}
   \langle \dop f,g \rangle
    &=
   \{ \overline{g}(+0)-\overline{g}(-0)\}\, f'(0)
   + \{\overline{g}^{\prime}(-0)-\overline{g}^{\prime}(+0)\}\, f(0)
   + \langle f,\dop_0^*g \rangle
   \\
   &=\{ \overline{g}(+0)-\overline{g}(-0)\}\, \frac{1}{2G(0)}  \langle \dop f, w_2\rangle
   + \{\overline{g}^{\prime}(-0)-\overline{g}^{\prime}(-0)\}\, \frac{1}{2} \langle \dop f, w_1 \rangle
   \\
   &\phantom{=}
   + \langle f,\dop_0^*g \rangle
   \\
   &=\{ \overline{g}(+0)-\overline{g}(-0)\}\, \frac{1}{2G(0)}  \langle f, \widetilde\dop^* w_2\rangle
   + \{\overline{g}^{\prime}(-0)-\overline{g}^{\prime}(+0)\}\, \frac{1}{2} \langle f, \widetilde\dop^* w_1\rangle
   \\
   &\phantom{=}
   + \langle f,\dop_0^*g \rangle.
\end{align*}
So by Lemma~\ref{lem:delta}, we obtain
\begin{align*}
   \widetilde \dop^* g
   = \frac{ \overline{g}(+0)-\overline{g}(-0)}{2G(0)} \widetilde\dop^* w_2
   + \frac{\overline{g'}(-0)-\overline{g'}(+0)}{2} \widetilde\dop^* w_1
   + \dop_0^*g,
\end{align*}
or, if we identify $H_-$ and $(H_+)'$,
\begin{align*}
   \widetilde \dop^* g
   = \{ \overline{g}(+0)-\overline{g}(-0) \} \delta'
   - \{\overline{g'}(+0)-\overline{g'}(-0) \} \delta
   + \dop_0^*g
   \ \in (H_+)'.
\end{align*}

Hence $\dop_0^*$ can be seen as a perturbation of $\widetilde\dop^*$:
\begin{align}
   \label{eq:AA}
   \dop_0^*g
   =
   \widetilde \dop^* g
   -\{ \overline{g}(+0)-\overline{g}(-0) \} \delta'
   + \{\overline{g'}(+0)-\overline{g'}(-0) \} \delta
   \ \in (H_+)'
\end{align}
for $g\in\mD(\dop_0^*)$.
Recall that the operators $\Bop_{\theta}$ from Section~\ref{sec:B} and $\cop_{K}$ from Section~\ref{sec:C} satisfy
$\Bop\subset\Bop_{\theta}\subset\dop_0^*$
and
$\cop\subset\cop_{K}\subset\dop_0^*$.
So we obtain the following:

\begin{itemize}

   \item
   Any function $g\in\mD(\cop_{K})$ with $K$ as in Subsection~\ref{subsec:C:classical} satisfies
   $g(-0) = g(+0)$ and $g'(-0) = g'(+0)$, hence
   \begin{equation*}
      \cop_{K}g = \widetilde A^* g
      = \dop_0^* g.
   \end{equation*}

   \item
   Any function $g\in\mD(\cop_{K})$ with $K$ as in Subsection~\ref{subsec:C:delta}
   and $\alpha\neq \pi/2$ satisfies
   $g(-0) = g(+0)$ and $g'(-0) - g'(+0) = 2\tan\alpha\ g(0)$, hence
   \begin{equation*}
      \cop_{K}g
      = \dop_0^* g
      = \widetilde\dop^* g - 2\tan\alpha\ g(0) \delta.
   \end{equation*}
   If we take $\theta$ such that $\cot\theta=-\sqrt{2}\tan\alpha$, we obtain
   \begin{equation*}
      \Bop_{\theta}g = \cop_{K}g
      = \sqrt{2}\cot\theta\ g(0) \delta +\widetilde\dop^* g.
   \end{equation*}
   Note that $\Bop_{\theta}$ as exactly one negative eigenvalue if $\theta\in(\pi/2+\alpha_A,\,\pi)$ and this eigenvalue decreases monotonically to $-\infty$ as $\theta\to\pi$, that is $\sqrt{2}\cot(\theta)\to-\infty$.

   \item
   Any function $g\in\mD(\cop_{K})$ with $K$ as in Subsection~\ref{subsec:C:deltaprime} satisfies
   $g'(-0)=g'(+0)$ and
   $g(+0)-g(-0)= -2\tan\alpha\ g'(0)$.
   Hence, for $\alpha\neq \pi/2$
   \begin{equation*}
      \cop_{K}g = \dop_0^* g
      = \widetilde\dop^* g + 2\tan\alpha\ g'(0) \delta'.
   \end{equation*}
\end{itemize}

\bigskip
\bigskip

\noindent
{\bf Acknowledgment.}
We are grateful to the
Fondo de Investigaciones de la Facultad de Ciencias de la Universidad de los Andes, Convocatoria 2014-1 para la Financiaci\'on de proyectos de Investigaci\'on Categor\'ia: Profesores De Planta, proyecto 
``Operadores lineales en espacios con product interno indefinido'',
for its financial support.

\bibliographystyle{alpha}
\bibliography{lit}

\end{document}